\documentclass[review,a4page]{elsarticle}
\usepackage[utf8]{inputenc}
\usepackage{graphicx}
\usepackage{srcltx}
\usepackage{eurosym}
\usepackage{mathtools}
\allowdisplaybreaks
\usepackage{enumerate}
\usepackage{amsmath}
\usepackage{amsfonts}
\usepackage{amssymb}
\usepackage{amsthm}
\usepackage{graphicx}
\usepackage{mathrsfs}
\usepackage{xcolor}
\usepackage{exscale}
\usepackage{latexsym}
\usepackage{esvect}
\usepackage[T1]{fontenc}

\usepackage{calc}
\usepackage[titletoc,toc,page]{appendix}
\usepackage{calc}
\usepackage[framemethod=tikz,xcolor=true]{mdframed}
\newmdenv[%
leftmargin=0.5cm,
backgroundcolor=yellow!10,%
roundcorner=5pt,%
tikzsetting={draw=blue, line width=2.0pt}%
]{SpecialText}%

\usepackage[colorlinks,plainpages=true,pdfpagelabels,hypertexnames=true,colorlinks=true,pdfstartview=FitV,linkcolor=blue,citecolor=red,urlcolor=black,backref=page]{hyperref}
\PassOptionsToPackage{unicode}{hyperref}
\PassOptionsToPackage{naturalnames}{hyperref}
\AtBeginDocument{%
	\hypersetup{citecolor=red}}
\usepackage{enumerate}
\usepackage[shortlabels]{enumitem}
\usepackage{bookmark}
\usepackage{wasysym}
\usepackage{esint}
\usepackage{setspace}
\usepackage[ddmmyyyy]{datetime}
\numberwithin{equation}{section}
\everymath{\displaystyle}
\usepackage[margin=2.2cm]{geometry}

\usepackage[capitalize,nameinlink]{cleveref}
\crefname{section}{Section}{Sections}
\crefname{subsection}{Subsection}{Subsections}
\crefname{condition}{Condition}{Conditions}
\crefname{hypothesis}{Hypothesis}{Hypothesis}
\crefname{assumption}{Assumption}{Assumptions}
\crefname{lemma}{Lemma}{Lemmas}
\crefname{claim}{Claim}{Claims}
\crefname{remark}{Remark}{Remarks}
\crefname{question}{Question}{Questions}

\crefformat{equation}{\textup{#2(#1)#3}}
\crefrangeformat{equation}{\textup{#3(#1)#4--#5(#2)#6}}
\crefmultiformat{equation}{\textup{#2(#1)#3}}{ and \textup{#2(#1)#3}}
{, \textup{#2(#1)#3}}{, and \textup{#2(#1)#3}}
\crefrangemultiformat{equation}{\textup{#3(#1)#4--#5(#2)#6}}%
{ and \textup{#3(#1)#4--#5(#2)#6}}{, \textup{#3(#1)#4--#5(#2)#6}}%
{, and \textup{#3(#1)#4--#5(#2)#6}}

\Crefformat{equation}{#2Equation~\textup{(#1)}#3}
\Crefrangeformat{equation}{Equations~\textup{#3(#1)#4--#5(#2)#6}}
\Crefmultiformat{equation}{Equations~\textup{#2(#1)#3}}{ and \textup{#2(#1)#3}}
{, \textup{#2(#1)#3}}{, and \textup{#2(#1)#3}}
\Crefrangemultiformat{equation}{Equations~\textup{#3(#1)#4--#5(#2)#6}}%
{ and \textup{#3(#1)#4--#5(#2)#6}}{, \textup{#3(#1)#4--#5(#2)#6}}%
{, and \textup{#3(#1)#4--#5(#2)#6}}

\crefdefaultlabelformat{#2\textup{#1}#3}
\newtheorem{theorem}{Theorem}[section]
\newtheorem{lemma}[theorem]{Lemma}

\newtheorem{corollary}[theorem]{Corollary}
\newtheorem{proposition}[theorem]{Proposition}

\newtheorem{remark}[theorem]{Remark}        

\numberwithin{equation}{section}
\newtheorem{question}[theorem]{Question}

\def\YYint#1#2#3{{\setbox0=\hbox{$#1{#2#3}{\iint}$}
		\vcenter{\hbox{$#2#3$}}\kern-.50\wd0}}

\def\XXint#1#2#3{{\setbox0=\hbox{$#1{#2#3}{\int}$}
		\vcenter{\hbox{$#2#3$}}\kern-.50\wd0}}

\makeatletter
\def\namedlabel#1#2{\begingroup
	\def\@currentlabel{#2}%
	\label{#1}\endgroup
}
\makeatother
\makeatletter
\newcommand{\rmh}[1]{\mathpalette{\raisem@th{#1}}}
\newcommand{\raisem@th}[3]{\hspace*{-1pt}\raisebox{#1}{$#2#3$}}
\makeatother

\newcommand{\redref}[2]{\texorpdfstring{\protect\hyperlink{#1}{\textcolor{black}{(}\textcolor{red}{#2}\textcolor{black}{)}}}{}}
\newcommand{\redlabel}[2]{\hypertarget{#1}{\textcolor{black}{(}\textcolor{red}{#2}\textcolor{black}{)}}}

\newcommand{\descref}[2]{\hyperref[#1]{\textcolor{black}{(}\textcolor{blue}{\bf #2}\textcolor{black}{)}}}

\newcommand{\dref}[2]{\hyperref[#1]{\textcolor{black}{(}\textcolor{blue}{\bf #2}\textcolor{black}{)}}}
\makeatletter
\g@addto@macro\normalsize{%
	\setlength\abovedisplayskip{3pt}
	\setlength\belowdisplayskip{3pt}
	\setlength\abovedisplayshortskip{1pt}
	\setlength\belowdisplayshortskip{3pt}
}
\makeatletter
\def\ps@pprintTitle{%
	\let\@oddhead\@empty
	\let\@evenhead\@empty
	\def\@oddfoot{}%
	\let\@evenfoot\@oddfoot}
\makeatother

\newcommand\RR{\mathbb{R}}

\newcommand\A{\mathbb{A}}
\newcommand\PP{\mathbb{P}}

\newcommand{\al}{\alpha}

\newcommand{\ve}{\varepsilon}

\newcommand{\om}{\omega}
\newcommand{\la}{\lambda}

\newcommand{\La}{\Lambda}

\DeclareMathOperator{\dv}{div}

\DeclareMathOperator{\loc}{loc}

\DeclareMathOperator*{\err}{\mathbb{E}_r}

\newcommand{\iprod}[2]{\left\langle #1 ,  #2\right\rangle}

\newcommand{\lbr}[1][(]{\left#1}
\newcommand{\rbr}[1][)]{\right#1}

\newcommand{\txt}[1]{\qquad \text{#1} \qquad}

\newcommand{\pa}{\partial}

\usepackage{bm}
\newcommand{\nbar}{\bm\bar{\nabla}}

\newcounter{whitney}
\refstepcounter{whitney}

\newcounter{ineqcounter}
\refstepcounter{ineqcounter}

\begin{document}
	\begin{frontmatter}
		
		\title{A note on H\"older regularity of weak solutions to linear elliptic equations}
		
		\author[myaddress]{Karthik Adimurthi\tnoteref{thanksfirstauthor}}
		
		\ead{karthikaditi@gmail.com and kadimurthi@tifrbng.res.in}

		\tnotetext[thanksfirstauthor]{Supported by the Department of Atomic Energy,  Government of India, under
			project no.  12-R\&D-TFR-5.01-0520}

		\address[myaddress]{Tata Institute of Fundamental Research, Centre for Applicable Mathematics,Bangalore, Karnataka, 560065, India}

		\begin{abstract}
			
			In this paper, we show that weak solutions of  $$-\dv \A(x)\nabla u = 0 \txt{where} A(x)= A(x)^T \,\, \text{and} \,\, \la |\zeta|^2 \leq \iprod{A(x)\zeta}{\zeta} \leq \La |\zeta|^2,$$
			and $\A(x) \equiv \A$ is a constant matrix are H\"older continuous $u \in C^{\alpha}_{\loc}$ with  $\al \geq \tfrac12 \lbr-(n-2) + \sqrt{(n-2)^2 + \tfrac{4(n-1)\la}{\La}} \rbr$. This implies that the example constructed by Piccinini - Spagnolo is sharp in the class of constant matrices $\A(x) \equiv \A$. The proof of H\"older regularity does not go through a reduction of oscillation type argument and instead is achieved through a monotonicity formula. 
			
			In the case of general matrices $\A(x)$, we obtain the same regularity under some additional hypothesis.  
		\end{abstract}
		
		
	\end{frontmatter}
	
	\begin{singlespace}
		\tableofcontents
	\end{singlespace}

	\section{Introduction}
	
	In this paper, we shall consider weak solutions $u \in W^{1,2}_{\loc}(\RR^n)$ solving 
	\begin{equation}\label{maineq}
		-\dv \A(x) \nabla u = 0 \txt{with} \la |\xi|^2 \leq \iprod{\A(x)\xi}{\xi} \leq \La |\xi|^2,
	\end{equation}
	where $\A(x)$ is assumed to be symmetric and $\la,\La \in (0,\infty)$ are any two fixed positive constants.  
	
	The H\"older regularity results were first proved in \cite{MR0093649,MR100158} and subsequently, a different proof was given in \cite{MR170091,MR159138,MR159139}.   All three approaches provided the tools to prove a much stronger Harnack inequality and H\"older regularity was deduced from this. In \cite{MR308945}, a sharp form of Harnack inequality was proved which implied that the H\"older exponent had exponential dependence on $\tfrac{\la}{\La}$. Since Harnack inequality obtained in \cite{MR308945} was sharp, this implied that the approaches developed in \cite{MR0093649,MR100158,MR170091} could at best give the H\"older exponent that had exponential dependence on $\tfrac{\la}{\La}$.

	In $\RR^2$, a new proof using a monotonicity formula which is an optimized version of the `hole filling' technique was given in \cite{MR361422} and they obtained the sharp H\"older exponent to be $\sqrt{\tfrac{\la}{\La}}$ and conjectured that in higher dimensions, the exponent should have the form
	\[
	\al \geq  \frac12 \lbr-(n-2) + \sqrt{(n-2)^2 + \frac{4(n-1)\la}{\La}} \rbr.
	\]
	In this regard, we also mention \cite{MR192330,MR798173} and references therein where the `hole filling' technique was extended to higher dimensional linear elliptic and parabolic equations using  Green's function estimates. 
	
	In this paper,  we prove the following theorem:
	%
	%
	%
	%
	%
	\begin{theorem}\label{mainthm}
		Any local weak solutions of \cref{maineq} with $\A(x) \equiv \A$  is H\"older continuous $u \in C_{\loc}^{\al}(\RR^n)$ with  $$\al \geq \frac12 \lbr-(n-2) + \sqrt{(n-2)^2 + \frac{4(n-1)\la}{\La}} \rbr.$$ 
	\end{theorem}
	
	\begin{corollary}
		Any local weak solutions of \cref{maineq} with a general $\A(x)$ and $\err \geq 0$ is H\"older continuous $u \in C_{\loc}^{\al}(\RR^n)$ with  $$\al \geq \frac12 \lbr-(n-2) + \sqrt{(n-2)^2 + \frac{4(n-1)\la}{\La}} \rbr.$$ 
		The quantity $\err$ is as obtained in \cref{poho_gen} and the regularity is obtained in \cref{rmka4.9}.
	\end{corollary}
	It suffices to prove the  following theorem instead:
	\begin{theorem}\label{mainthm2}
		Let $S_1$ be the sphere of unit radius in $\RR^n$ and $B_1$ be the unit ball in $\RR^n$. Then, any weak solution of \cref{maineq} satisfies
		\[
		\int_{S_1} \iprod{\A\nabla u}{\nabla u} \,d\sigma  \geq \sqrt{(n-2)^2 + \frac{4(n-1)\la}{\La}}\int_{B_1}\iprod{\A\nabla u}{\nabla u}\,dx.
		\]
	\end{theorem}
	
	\begin{remark}
		In the rest of the paper, we will assume $\A(x)$ and $u$ are smooth and obtain a priori estimates. Such an assumption can be made to hold by a standard approximation procedure, see \cite[Section 4]{delellis2017masterpieces} for the details. The proof also shows that for general matrices $\A(x)$, we can obtain the same regularity as \cref{mainthm} provided $\err \geq 0$, where $\err$ is as obtained in \cref{poho_gen}. 
	\end{remark}
	
	\section{Ideas of Piccinini-Spagnolo revisited}
	In \cite{MR361422}, the authors proved the following result: 
	\begin{theorem}\label{thm2.1}
		When $n=2$, weak solutions of \cref{maineq} are $C^{\alpha}_{\loc}$ regular with $\al = \sqrt{\tfrac{{\la}}{{\La}}}$. In particular, the following estimate holds for any $r >0$:
		\begin{equation}\label{a2.1}
			\int_{B_r} \iprod{\A(x)\nabla u}{\nabla u}\, dx \leq \frac{r{\sqrt{\frac{\La}{\la}}}}{2}\int_{S_r} \iprod{\A(x)\nabla u}{\nabla u}\, d\sigma,
		\end{equation}
		where $B_r$ is the ball of raidus $r$ and $S_r= \partial B_r$  is the boundary of $B_r$. 
	\end{theorem}
	
	\begin{proof}
		We shall change to polar co-ordinates and use the following notation:
		\[
		\begin{bmatrix}
			x\\y 
		\end{bmatrix}
		=\begin{bmatrix}
			r \cos (\theta) \\
			r \sin(\theta) 
		\end{bmatrix},
		\quad 
		\mathbb{J}(\theta)=\begin{bmatrix}
			\cos(\theta) & -\sin(\theta)\\
			\sin(\theta) & \cos(\theta)
		\end{bmatrix},
		\quad 
		\PP(x) = \mathbb{J}(\theta)^\top \A(x) \mathbb{J}(\theta), \quad 
		\underbrace{\begin{bmatrix}
				u_x \\ u_y
		\end{bmatrix}}_{:=\nabla u}
		=\mathbb{J}(\theta)\underbrace{\begin{bmatrix}
				u_{r} \\ \tfrac{1}{r} u_{\theta}
		\end{bmatrix}}_{:=\nbar u} := \begin{bmatrix}
			u_{N} \\ u_{T}
		\end{bmatrix}.
		\]
		Following \cite{MR361422}, we have
		\begin{equation}\label{a2.2}
			\begin{array}{rcl}
				g(r) &:=&	\int_{B_r} \iprod{\A(x)\nabla u}{\nabla u}\, dx = \int_{S_r} (u-k) \iprod{\PP(x)\nbar u}{\vec{e_1}}\, d\sigma = \int_{S_r} (u-k) \lbr p_{11} u_{\rho} + p_{12} \tfrac{1}{r}u_{\theta} \rbr d\sigma, \\
				g'(r) &:=& \int_{S_r} \iprod{\PP(x)\nbar u}{e_1}\, d\sigma,
			\end{array}
		\end{equation}
		where to obtain the first equality in the definition of $g(r)$, we made use of the fact that $u$ solves \cref{maineq}.

		With the choice $k = \fint_{S_r} u \,d\sigma$ and $L:= \tfrac{\La}{\la}$, we shall estimate $g(r)$ as follows:
		\begin{equation}\label{a2.3}
			\begin{array}{rcl}
				g(r) & \overset{\redlabel{a2.3a}{a}}{\leq} & \lbr \int_{S_r} p_{11} (u-k)^2 \,d\sigma \rbr^{\frac{1}{2}} \lbr \int_{S_r}  \lbr\sqrt{p_{11}} u_{r} + \frac{p_{12}}{\sqrt{p_{11}}} \frac{1}{r} u_{\theta}\rbr^2 \,d\sigma\rbr^{\frac{1}{2}}\\
				&\overset{\redlabel{a2.3b}{b}}{\leq} & r \sqrt{L} \lbr \la  \int_{S_r}  u_T^2 \,d\sigma \rbr^{\frac{1}{2}} \lbr \int_{S_r}  \lbr\sqrt{p_{11}} u_{N} + \frac{p_{12}}{\sqrt{p_{11}}} u_T\rbr^2 \,d\sigma\rbr^{\frac{1}{2}}\\
				&\overset{\redlabel{a2.3c}{c}}{\leq} &\frac{r\sqrt{L}}{2} \lbr[[] \int_{S_r}  \lbr p_{22} - \frac{p_{12}^2}{p_{11}}\rbr u_T^2 \,d\sigma  + \int_{S_r}  \lbr p_{11} u_N^2 + \frac{p_{12}^2}{p_{11}} u_T^2 + 2p_{12} u_N u_T  \rbr \,d\sigma\rbr[]]\\
				&= & \frac{r\sqrt{L}}{2} \int_{S_r} \iprod{P \nbar u}{\nbar u}\,d\sigma =  \frac{r\sqrt{L}}{2}  g'(r),
			\end{array}
		\end{equation}
		where to obtain \redref{a2.3a}{a}, we applied H\"older's inequality, to obtain \redref{a2.3b}{b}, we applied \cref{poincare} along with the bound $\la \leq p_{11}(x) \leq \La$ and finally  to obtain \redref{a2.3c}{c}, we made use of the matrix inequality $\la \leq p_{22}(x) - \frac{p_{12}^2(x)}{p_{11}(x)} \leq \La$ along with Young's inequality. Recalling \cref{a2.2}, the estimate in \cref{a2.3} is exactly the claim in \cref{a2.1} and this completes the proof of the theorem.
		
	\end{proof}
	\section{Proof of \texorpdfstring{\cref{mainthm2}}. when \texorpdfstring{$\A(x) = \mathbb{I}$}. is identity matrix - case of harmonic functions}
	
	Let us first recall the Poincar\'e inequality on spheres (see \cite[Chapter IV, Section 2]{MR0304972} for the details):
	\begin{proposition} \label{poincare}
		Let $u \in W^{1,2}_{\loc}(S_r)$ where $S_r$ is the unit sphere in $\RR^n$ with $n \geq 2$ and $\nbar u = (u_N,u_{T})$ be the tangential and normal derivative of $u$. Then the following sharp form of Poincar\'e's inequality holds:
		\[
		\int_{S_r} ( u- k)^2 \,d\sigma \leq  \lbr \frac{r^2}{n-1}\rbr\int_{S_r} |u_T|^2 \,d\sigma,
		\]
		where $k= \fint_{S_r} u \,d\sigma$.
		
	\end{proposition}
	
	We shall also recall the well known Poho\v{z}aev identity for harmonic functions obtained in \cite{MR0192184}.
	\begin{proposition}\label{pohozaev}
		Let $u$ to be a smooth, weak solution of $-\Delta u =0$ (i.e., $u$ is locally  harmonic), then the following identity holds:
		\begin{equation*}
			\int_{S_1} |u_T|^2 \,d\sigma = \int_{S_1} |u_N|^2   \,d\sigma + (n-2) \int_{B_1} |\nabla u|^2 \ dx.
		\end{equation*}
	\end{proposition}
	\begin{proof}
		Simple calculations implies
		\begin{equation}\label{aa3.1}
			\Delta u \iprod{\nabla u}{x} = \dv(\nabla u \iprod{\nabla u}{x}) - \frac{1}{2} \dv (|\nabla u|^2 x) + \frac{n-2}{2} |\nabla u|^2.
		\end{equation}
		Using $-\Delta u =0$ along with integrating \cref{aa3.1} over $B_1$, we have the following identity
		\begin{equation*}
			\begin{array}{rcl}
				0 = \int_{B_1}\Delta u \iprod{\nabla u}{x} \, dx  & = &  \int_{S_1} \iprod{\nabla u}{x}^2 \,d\sigma - \frac12 \int_{S_1} |\nabla u|^2 \iprod{x}{x} \,d\sigma + \frac{n-2}{2} \int_{B_1} |\nabla u|^2 \,dx \\
				& = &\int_{S_1} |u_N|^2 \,d\sigma - \frac{1}2 \int_{S_1} (|u_T|^2 + |u_N|^2)  \,d\sigma + \frac{n-2}{2} \int_{B_1} |\nabla u|^2 \,dx,
			\end{array}
		\end{equation*}
		which completes the proof. 
	\end{proof}

	\subsection{Some discussion}

	If we consider solutions of $-\Delta u =0$ in $\RR^n$ (i.e., $\A(x) = I$) for $n \geq 3$ and try to  follow the calculations of \cref{thm2.1}, then we get 
	\begin{equation}\label{a3.1}
		\int_{B_r} \iprod{\nabla u}{\nabla u}\, dx \leq \frac{r}{2\sqrt{n-1}}\int_{S_r} \iprod{\nabla u}{\nabla u}\, d\sigma,
	\end{equation}
	which fails to prove Lipschitz regularity (i.e., the constant in \cref{a3.1} should be $\tfrac{r}{n}$) when $n=3$ and for $n \geq 7$, we have $2 \sqrt{n-1} < n-2$ and thus no regularity follows. 
	
	On the other hand, if we consider the simple  example of harmonic functions  which satisfy $|\nabla u| = C$ for some constant $C$, then we have 
	\[
	\frac{\int_{B_r} \iprod{\nabla u}{\nabla u}\, dx}{\int_{S_r} \iprod{\nabla u}{\nabla u}\, d\sigma} = \frac{\omega_nr^n}{n \om_{n}r^{n-1}} = \frac{r}{n},
	\]
	which implies Lipschitz regularity. But the above calculations shows that the factor $\tfrac{1}{n}$ is coming from the ratio of $|B_r|$ and $|S_r|$ which is a property of the fact that we are in $\RR^n$ and not necessarily from the property of the equation and solution. 
	
	This suggest that some additional cancellations should hold that can further improve \cref{a3.1} and these cancellations should be obtained from the fact that we are in $\RR^n$. We shall formalize this observation for harmonic functions in the following theorem:
	
	\begin{theorem}\label{thm3.1}
		Let $u$ be a local weak solution of $-\Delta u = 0$, then  we have 
		\[
		\int_{B_r} \iprod{\nabla u}{\nabla u}\, dx \leq \frac{r}{n}\int_{S_r} \iprod{\nabla u}{\nabla u}\, d\sigma,\]
		holds for all $r >0$ and in particular $u \in C^{0,1}_{\loc}$.
	\end{theorem}
	
	\subsection{Discussion of the two proofs}
	
	We shall give two proofs of \cref{thm3.1}, both very similar, but have some crucial differences which are detailed in the following observations: 
	\begin{enumerate}[(i)]
		\item In the first proof, we need to make use of \cref{a3.7} to obtain \cref{a3.7b}. This does not require equality to hold in \cref{a3.7}, but it suffices if the following inequality holds:
		\begin{equation}\label{ineqlow}
			\int_{S_1} |u_T|^2 \,d\sigma \leq \int_{S_1} |u_N|^2   \,d\sigma + (n-2) \int_{B_1} |\nabla u|^2 \ dx.
		\end{equation}
		\item In the first proof, we need to make use of \cref{a3.7} a second time in \cref{a3.8aa}. This again does not require equality to hold in \cref{a3.7}, but it suffices if the following inequality holds:
		\begin{equation}\label{ineqhigh}
			\int_{S_1} |u_T|^2 \,d\sigma \geq \int_{S_1} |u_N|^2   \,d\sigma + (n-2) \int_{B_1} |\nabla u|^2 \ dx.
		\end{equation}
		\item In particular, the first proof requires \cref{ineqhigh} and \cref{ineqlow} to both hold, which is the equality form of the identity obtained in \cref{pohozaev}.
		
		\item In the second proof, we only need to use the Poho\v{z}aev identity \cref{a3.9} once to obtain \cref{a3.10}. In particular, the second proof works if we only  had the following one sided inequality:
		\begin{equation*}
			\int_{S_1} |\nabla u|^2 \,d\sigma \geq 2\int_{S_1} |u_N|^2   \,d\sigma + (n-2) \int_{B_1} |\nabla u|^2 \ dx.
		\end{equation*}
		
	\end{enumerate}
	
	\subsection{First proof}
	We shall use the notation $\nbar u = (u_N, u_T)$  where $u_N = \iprod{\nbar u}{\vec{e_1}}$  and $u_T = (\iprod{\nbar u}{\vec{e_2}},\ldots, \iprod{\nbar u}{\vec{e_n}})$ are the normal and tangential derivatives on the sphere respectively in polar coordinates.  We can also take $r=1$ since the required estimate can be obtained by scaling. 
	
	\begin{remark} We will be deliberately careless in switching between $\nabla u$ and $\nbar u$, i.e., cartesian and polar coordinates in the proof, so as to better present the ideas of the proof, noting that this switch is just a transformation by orthogonal matrices and hence does not affect any of the calculations.  
	\end{remark} 
	The proof will now proceed in several steps:
	\begin{description}[leftmargin=*]
		\item[Step 1:] 	Since $u$ is harmonic, we see that 
		\begin{equation}\label{a3.2}
			\begin{array}{rcl}
				\int_{B_1} \iprod{\nabla u}{\nabla u} \,dx & = &  \int_{S_1} (u-k) \iprod{\nabla u}{\vec{n}} \,d\sigma = \int_{S_1} (u-k) \iprod{\nbar u}{\vec{e_1}} \,d\sigma \\
				& \overset{\redlabel{a3.2a}{a}}{\leq} &  \lbr \frac{1}{2(n-1)} \int_{S_1} |u_{T}|^2 \,d\sigma + \frac{1}{2} \int_{S_1} |u_N|^2\,d\sigma\rbr,\\		 
			\end{array}
		\end{equation}
		where to obtain \redref{a3.2a}{a}, we applied Young's inequality followed by \cref{poincare} with $k = \fint_{S_1} u \,d\sigma$. 
		\item[Step 2:] We recall  Poho\v{z}aev identity from \cref{pohozaev}:
		\begin{equation}\label{a3.7}
			\int_{S_1} |u_T|^2 \,d\sigma = \int_{S_1} |u_N|^2   \,d\sigma + (n-2) \int_{B_1} |\nabla u|^2 \ dx.
		\end{equation}
		\item[Step 3:] We shall substitute \cref{a3.7} into \cref{a3.2} and obtain
		\begin{equation}\label{a3.7b}
			\int_{B_1} |\nabla u|^2 \ dx \leq \frac{1}{2(n-1)} \lbr \int_{S_1} |u_N|^2 \, d\sigma + (n-2) \int_{B_1} |\nabla u|^2 \, dx \rbr + \frac{1}{2} \int_{S_1} |u_N|^2 \,d\sigma,
		\end{equation}
		which after simplification, becomes
		\begin{equation}\label{a3.8}
			\int_{B_1} |\nabla u|^2 \ dx  \leq  \int_{S_1} |u_N|^2 \,d\sigma.
		\end{equation}
		
		\item[Step 4:]
		Let us add $\int_{S_1} |u_N|^2 \,d\sigma$ to both sides of \cref{a3.7} and then make use of \cref{a3.8} to get
		\begin{equation}\label{a3.8aa}
			\begin{array}{rcl}
				\int_{S_1} |\nabla u|^2 \,d\sigma = \int_{S_1} (|u_T|^2 +|u_N|^2) \,d\sigma & \overset{\cref{a3.7}}{=} & 2\int_{S_1} |u_N|^2   \,d\sigma + (n-2) \int_{B_1} |\nabla u|^2 \ dx\\
				& \overset{\cref{a3.8}}{\geq} &2 \int_{B_1} |\nabla u|^2 \ dx  + (n-2)\int_{B_1} |\nabla u|^2 \ dx \\ &=& n \int_{B_1}|\nabla u|^2 \, dx.
			\end{array}
		\end{equation}
	\end{description}
	This completes the proof of the theorem.
	
	\subsection{Second proof}
	Using the same notation as the first proof, we will give a slightly modified proof of \cref{thm3.1} which will proceed in several steps:
	\begin{description}[leftmargin=*]
		\item[Step 1:] 	Since $u$ is harmonic, we see that 
		\begin{equation}\label{a3.8b}
			\begin{array}{rcl}
				\int_{B_1} \iprod{\nabla u}{\nabla u} \,dx & = &  \int_{S_1} (u-k) \iprod{\nabla u}{\vec{n}} \,d\sigma = \int_{S_1} (u-k) \iprod{\nbar u}{\vec{e_1}} \,d\sigma \\
				& \overset{\redlabel{a3.8ba}{a}}{\leq} &  \lbr \frac{1}{2(n-1)} \int_{S_1} |u_{T}|^2 \,d\sigma + \frac{1}{2} \int_{S_1} |u_N|^2\,d\sigma\rbr\\
				& \overset{\redlabel{a3.8bb}{b}}{=} &\lbr \frac{1}{2(n-1)}\int_{S_1} |\nabla u|^2 \,d\sigma + \lbr \frac{1}{2}-\frac{1}{2(n-1)}\rbr \int_{S_1} |u_N|^2\,d\sigma\rbr,
			\end{array}
		\end{equation}
		where to obtain \redref{a3.8ba}{a}, we applied Young's inequality followed by \cref{poincare} with $k = \fint_{S_1} u \,d\sigma$ and to obtain \redref{a3.8bb}{b}, we add and subtract $\frac{1}{2(n-1)}\int_{S_1} |u_N|^2 \,d\sigma$. 
		\item[Step 2:] We rewrite the  Poho\v{z}aev identity from \cref{pohozaev} as:
		\begin{equation}\label{a3.9}
			\int_{S_1} |\nabla u|^2 \,d\sigma = 2\int_{S_1} |u_N|^2   \,d\sigma + (n-2) \int_{B_1} |\nabla u|^2 \ dx.
		\end{equation}
		\item[Step 3:] We shall substitute \cref{a3.9} into \cref{a3.8b} to get
		\begin{equation}\label{a3.10}
			\int_{B_1} |\nabla u|^2 \ dx \leq \frac{1}{2(n-1)}\int_{S_1} |\nabla u|^2 \, d\sigma + \lbr \frac{1}{2}-\frac{1}{2(n-1)}\rbr \frac12\lbr \int_{S_1} |\nabla u|^2 \,d\sigma - (n-2) \int_{B_1}|\nabla u|^2 \,dx  \rbr,
		\end{equation}
		which after simplification, becomes 
		\begin{equation*}
			\frac{n^2}{4(n-1)}\int_{B_1} |\nabla u|^2 \ dx  \leq  \frac{n}{4(n-1)}\int_{S_1} |\nabla u|^2 \,d\sigma.
		\end{equation*}
	\end{description}
	This completes the proof of the lemma.

	\section{Proof of \texorpdfstring{\cref{mainthm2}}. when \texorpdfstring{$\A(x) \equiv \A$}. has constant entries}
	
	The first lemma we generalize is the Poho\v{z}aev identity from \cref{pohozaev} to variable coefficient operators of the form considered in \cref{maineq}. This is well known in literature and we present all the details for sake of completeness. 
	\begin{lemma}\label{poho_gen}
		Let $u$ be a weak solution to \cref{maineq} and assume both $u$ and $\A(x)$ are sufficiently smooth. Then the following Poho\v{z}aev type identity holds:
		\begin{equation}\label{eq4.1}
			\begin{array}{rcl}
				\int_{S_1} \iprod{\A(x)\nabla u}{\nabla u} \iprod{\A(x)x}{x} \,d\sigma & = & 	2 \int_{S_1} \iprod{\A(x)\nabla u}{x}^2\,d\sigma 
				+ \int_{B_1} \lbr \sum_{i}^n a_{ii}(x)\rbr \iprod{\A(x)\nabla u }{\nabla u}\,dx \\
				&& + \sum_{i,j,k,l}  \int_{B_1} x_j (\pa_{x_i} a_{ij}) a_{kl}u_{x_k}u_{x_l}\,dx 
				- 2\int_{B_1} \iprod{\A(x)\nabla u}{\A(x)\nabla u}\,dx\\
				&& - 2\sum_{i,j,k,l}\int_{B_1}  u_{x_i} a_{ij} u_{x_k}  x_l (\pa_{x_j} a_{kl})\,dx \\
				&&
				+ \sum_{i,j,m,n}\int_{B_1} a_{ij}x_ju_{x_m}u_{x_n} (\pa_{x_i} a_{mn})\,dx\\
				& = & 	2 \int_{S_1} \iprod{\A(x)\nabla u}{x}^2\,d\sigma 
				+ \int_{B_1} \lbr \sum_{i}^n a_{ii}(x)\rbr \iprod{\A(x)\nabla u }{\nabla u}\,dx \\
				&& 
				- 2\int_{B_1} \iprod{\A(x)\nabla u}{\A(x)\nabla u}\,dx + \err,
			\end{array}
		\end{equation}
		where we have denoted 
		\[
		\err  := \sum_{i,j,k,l}  \int_{B_1} x_j (\pa_{x_i} a_{ij}) a_{kl}u_{x_k}u_{x_l}\,dx - 2\sum_{i,j,k,l}\int_{B_1}  u_{x_i} a_{ij} u_{x_k}  x_l (\pa_{x_j} a_{kl})\,dx+ \sum_{i,j,m,n}\int_{B_1} a_{ij}x_ju_{x_m}u_{x_n} (\pa_{x_i} a_{mn})\,dx.
		\]
	\end{lemma} 
	\begin{remark}
		Note that all the terms of $\err$ contain derivates of the entries of $\A$ and does not contain any terms from the Hessian matrix $[D^2u]$. In particular, \cref{eq4.1} already captures the cancellations of the form $\iprod{x}{[D^2u] \nabla u} = \iprod{\nabla u}{[D^2u] x}$ which is crucially used in the case of  harmonic functions to prove \cref{pohozaev}.
	\end{remark}
	\begin{proof}[Proof of \cref{poho_gen}]
		Simple calculations give
		\begin{equation}\label{a4.2}
			\begin{array}{l}
				\dv (\A\nabla u \iprod{x}{\A\nabla u}) = \dv (\A\nabla u) \iprod{x}{\A\nabla u} + \iprod{\A\nabla u}{\A \nabla u} + \iprod{\A\nabla u}{[D(\A\nabla u)] x},\\
				\dv(\A x \iprod{\A\nabla u}{\nabla u}) = \dv(\A x) \iprod{\A\nabla u}{\nabla u} + \iprod{\A x}{[D(\A\nabla u)]\nabla u} + \iprod{\A x}{[D^2u]\A\nabla u},
			\end{array}
		\end{equation}
		where $D(\A \nabla u)$ is a matrix of the form 
		\begin{equation}\label{a4.3}
			(D(\A\nabla u))_{ij} = \sum_{k=1}^n (\pa_{x_i} a_{jk}) u_{x_k} + (D^2u \A)_{ij} =: (D(A)(\nabla u))_{ij} + ([D^2u] \A)_{ij},
		\end{equation}
		and $([D^2 u] \A)$ is the matrix  product of  $[D^2u]\cdot \A$. 
		Using \cref{a4.3} in \cref{a4.2} and noting that $\dv(\A\nabla u)=0$, we get
		\begin{equation}\label{a4.4}
			\begin{array}{l}
				\dv (\A\nabla u \iprod{x}{\A\nabla u}) =  \iprod{\A\nabla u}{\A \nabla u} + \iprod{\A\nabla u}{[D(\A)(\nabla u)] x} + \iprod{\A\nabla u}{[D^2 u] \A x},\\
				\dv(\A x \iprod{\A\nabla u}{\nabla u}) = \dv(\A x) \iprod{\A\nabla u}{\nabla u} + \iprod{\A x}{[D(\A)(\nabla u)]\nabla u} + 2\iprod{\A x}{[D^2u]\A\nabla u},
			\end{array}
		\end{equation}
		Subtracting twice the first equation in  \cref{a4.4} from the second equation in \cref{a4.4} in order to cancel the terms containing $[D^2u]$ (since $\iprod{\A x}{[D^2u]\A\nabla u} = \iprod{\A \nabla u}{[D^2u]\A x}$), we get
		\begin{multline}\label{a4.5}
			2\dv (\A\nabla u \iprod{x}{\A\nabla u}) - \dv(\A x \iprod{\A\nabla u}{\nabla u})\\ = 2\iprod{\A\nabla u}{\A \nabla u} + 2\iprod{\A\nabla u}{[D(\A)(\nabla u)] x} - \dv(\A x) \iprod{\A\nabla u}{\nabla u} - \iprod{\A x}{[D(\A)(\nabla u)]\nabla u}.
		\end{multline}
		We see that $\dv (\A x) = \sum_{i=1}^n a_{ii} + \sum_{i,j} x_j \pa_{x_i} a_{ij}$
		which we substitute in \cref{a4.5} and integrate over $B_1$ to get
		\begin{equation*}
			\begin{array}{rcl}
				2 \int_{B_1}\dv (\A\nabla u \iprod{x}{\A\nabla u}) \,dx & = &  \int_{B_1} \dv(\A x \iprod{\A\nabla u}{\nabla u})\,dx + 2 \int_{B_1}\iprod{\A\nabla u}{\A \nabla u} \, dx \\
				&& + 2\int_{B_1}\iprod{\A\nabla u}{[D(\A)(\nabla u)] x}\,dx - \int_{B_1}\iprod{\A x }{[D(\A)(\nabla u)] \nabla u}\,dx\\
				&& - \int_{B_1} \lbr \sum_{i}^n a_{ii}\rbr \iprod{\A\nabla u }{\nabla u}\,dx - \int_{B_1} \lbr \sum_{i,j} x_j \pa_{x_i}a_{ij}\rbr\iprod{\A\nabla u }{\nabla u} \,dx.
			\end{array}
		\end{equation*}
		We can now apply the divergence theorem to complete the proof of the lemma. 
	\end{proof}
	\begin{corollary}
		\label{poho_const_corr}
		In the case $\A(x) \equiv \A$ is a constant matrix, then $\err = 0$ and \cref{eq4.1} becomes 
		\begin{equation}\label{eq4.1_corr}
			\begin{array}{rcl}
				\int_{S_1} \iprod{\A\nabla u}{\nabla u} \iprod{\A x}{x} \,d\sigma & = & 	2 \int_{S_1} \iprod{\A\nabla u}{x}^2\,d\sigma 
				+ \int_{B_1} \lbr \sum_{i}^n a_{ii}\rbr \iprod{\A\nabla u }{\nabla u}\,dx \\
				&& 
				- 2\int_{B_1} \iprod{\A\nabla u}{\A\nabla u}\,dx.
			\end{array}
		\end{equation}
	\end{corollary}
	
	Let us now prove a matrix inequality using the Schur complement of a block matrix that will be the higher dimensional replacement of the cancellation used in  \redref{a2.3c}{c} of \cref{a2.3}.
	\begin{lemma}
		\label{schur_lemma}
		Let $\PP$ be the matrix similar to $\A$ in polar coordinates which we write as 
		\[\PP=\left( \begin{array}{c|c} p_{11} & P_{12}\\
			\hline
			P_{21} & \PP_{22} \end{array} \right),
		\]
		where $p_{11}$ is a scalar, $P_{12} = P_{21}^{\top}$ is $(n-1)$ vector and $\PP_{22}$ is $(n-1) \times (n-1)$ matrix. Then the following inequality holds:
		\begin{equation} \label{a4.8}   \la p_{11}|\xi|^2 \leq p_{11} \iprod{\PP_{22}\xi}{\xi} - \iprod{P_{12}}{\xi}^2 \qquad \text{for every} \,\, \xi \in \RR^{n-1}\setminus \{0\}.
		\end{equation}
	\end{lemma}
	\begin{proof}
		We have
		\[
		\underbrace{\begin{pmatrix}1 & 0\\-P_{21}p_{11}^{-1}&I\end{pmatrix}}_{:=\mathbb{S}}
		\underbrace{\begin{pmatrix}p_{11} & P_{12}\\P_{21} & P_{22}\end{pmatrix}}_{=\PP} \underbrace{\begin{pmatrix}1 & -p_{11}^{-1}P_{12}\\0 &I\end{pmatrix}}_{=\mathbb{S}^{\top}}
		= 
		\underbrace{\begin{pmatrix}p_{11} & 0\\0&P_{22} - P_{21}p_{11}^{-1}P_{12}\end{pmatrix}}_{:=\mathbb{M}}.
		\] 
		From this, we see that if we consider the $\RR^{n}$ vector $\begin{pmatrix}
			0\\ \xi 
		\end{pmatrix}$, then using \cref{maineq}, we have 
		\begin{equation*}
			\begin{array}{rcl}
				\la|\xi |^2\leq 	\la \frac{\iprod{P_{12}}{\xi }^2}{p_{11}^2}+\la|\xi |^2 &\leq& \iprod{\PP \begin{pmatrix}
						-p_{11}^{-1}\iprod{P_{12}}{\xi } \\\xi 
				\end{pmatrix}}{\begin{pmatrix}
						-p_{11}^{-1}\iprod{P_{12}}{\xi } \\\xi 
				\end{pmatrix}}	 \\
				&=&  \iprod{\mathbb{S} \mathbb{\PP} \mathbb{S}^{\top}\begin{pmatrix}
						0 \\\xi 
				\end{pmatrix}}{\begin{pmatrix}
						0 \\\xi 
				\end{pmatrix}}	 
				=  \iprod{ \mathbb{M} \begin{pmatrix}
						0 \\\xi 
				\end{pmatrix}}{\begin{pmatrix}
						0 \\\xi 
				\end{pmatrix}}	 \\
				& = & \iprod{\begin{pmatrix}p_{11} & 0\\0&P_{22} - P_{21}p_{11}^{-1}P_{12}\end{pmatrix}\begin{pmatrix}
						0 \\\xi 
				\end{pmatrix} }{\begin{pmatrix}
						0 \\\xi 
				\end{pmatrix}} \\
				& = & \iprod{(\PP_{22} - P_{21}p_{11}^{-1}P_{12})\xi }{\xi }
			\end{array}
		\end{equation*}
		In particular, we have 
		\begin{equation*}
			\la p_{11}|\xi|^2 \leq p_{11} \iprod{\PP_{22}\xi}{\xi} - \iprod{P_{12}}{\xi}^2 \qquad \text{for every} \,\, \xi \in \RR^{n-1},
		\end{equation*}
		which completes the proof of the lemma. 
	\end{proof}

	We are now ready to prove \cref{mainthm2} in the case $\A(x) \equiv \A$ is a constant matrix:
	
	\begin{proof}[Proof of \cref{mainthm2}]
		We make a note about the notation used, we will use $\A$ and $\nabla u$ to denote quantities in Cartesian coordinates and $\PP$ and $\nbar u$ to denote analogue quantities in Polar coordinates. It is easy to see that the matrices $\A$ and $\PP$ are similar to each other  via the spherical orthogonal change of variables, see the proof of \cref{thm2.1} for this in $\RR^2$.

		The proof follows in several steps:
		\begin{description}
			\item[Step 1:]  Using integration by parts and noting that $-\dv \A(x)\nabla u=0$, we have 
			\begin{equation*}
				\int_{B_1} \iprod{\A(x)\nabla u}{\nabla u}\, dx = \int_{S_1} (u-k) \iprod{\A(x)\nabla u}{\vec{n}} \,d\sigma,
			\end{equation*}
			where we take $k = \fint_{S_1} u \,d\sigma$. This can be further estimated after changing to polar coordinates to get
			\begin{equation}\label{eqa4.4}
				\begin{array}{rcl}
					\int_{S_1} (u-k) \iprod{\A(x)\nabla u}{\vec{n}} \,d\sigma &\leq &\frac{\ve}{2} \int_{S_1} (u-k)^2 \,d\sigma + \frac{1}{2\ve} \int_{S_1} \iprod{\PP(x) \nbar u}{\vec{e_1}}^2 \,d\sigma \\
					&\overset{\redlabel{a4.4a}{a}}{\leq} &\frac{\ve}{2(n-1)} \int_{S_1} |u_T|^2 \,d\sigma + \frac{1}{2\ve} \int_{S_1} \iprod{\PP(x) \nbar u}{\vec{e_1}}^2 \,d\sigma,
				\end{array}
			\end{equation}
			where to obtain \redref{a4.4a}{a}, we made use of \cref{poincare}. 
			\item[Step 2:] We will apply \cref{a4.8} with $\xi = u_T$ to get
			\begin{equation}\label{a4.13}
				\begin{array}{rcl}
					p_{11}\la |u_T|^2 & \leq &  p_{11}\iprod{\PP_{22}u_T}{u_T} - \iprod{P_{12}}{u_T}^2\\
					& = & p_{11}\iprod{\PP_{22}u_T}{u_T} - \iprod{P_{12}}{u_T}^2   +\iprod{\PP \nbar u}{\vec{e_1}}^2 - \iprod{\PP \nbar u}{\vec{e_1}}^2\\
					& = & p_{11} \iprod{\PP \nbar u}{\nbar u}- \iprod{\PP \nbar u}{\vec{e_1}}^2\\
				\end{array}
			\end{equation}

			Using \cref{a4.13} along with $\la \leq p_{11} \leq \La$, we estimate  $\int_{S_1} |u_T|^2 \,d\sigma$ in \cref{eqa4.4} to get
			\begin{equation}\label{a4.14}
				\int_{B_1} \iprod{\PP(x)\nbar u}{\nbar u}\, dx
				\leq \frac{\ve}{2(n-1)\la} \int_{S_1} \iprod{\PP \nbar u}{\nbar u} \,d\sigma + \lbr \frac{1}{2\ve}- \frac{\ve}{2(n-1)\la\La}\rbr \int_{S_1} \iprod{\PP(x) \nbar u}{\vec{e_1}}^2 \,d\sigma.
			\end{equation}
			\item[Step 3:] We recall Poho\v{z}aev identity from \cref{eq4.1_corr} noting that $\sum_{i}^n a_{ii}  = \sum_{i}^n p_{ii}$ to get
			\begin{equation}\label{eq4.15}
				\begin{array}{rcl}
					\int_{S_1} \iprod{\PP\nbar u}{\nbar u} \iprod{\PP \vec{e_1}}{\vec{e_1}} \,d\sigma & = & 	2 \int_{S_1} \iprod{\PP\nbar u}{\vec{e_1}}^2\,d\sigma 
					+ \int_{B_1} \lbr \sum_{i}^n p_{ii}\rbr \iprod{\PP \nbar u }{\nbar u}\,dx \\
					&& 
					- 2\int_{B_1} \iprod{\PP\nabla u}{\PP\nabla u}\,dx + \err.
				\end{array}
			\end{equation}
			\item[Step 4:] Let us restrict $\ve \leq \sqrt{(n-1)\la\La}$ which is equivalent to $\lbr \frac{1}{2\ve}- \frac{\ve}{2(n-1)\la\La}\rbr \geq 0$.  Combining \cref{eq4.15} and \cref{a4.14} gives \begin{multline*}
				\int_{B_1} \iprod{\PP(x)\nbar u}{\nbar u}\, dx
				\leq   \frac{\ve}{2(n-1)\la} \int_{S_1} \iprod{\PP \nbar u}{\nbar u} \,d\sigma
				+ \lbr \frac{1}{2\ve}- \frac{\ve}{2(n-1)\la\La}\rbr\times \\
				\times \frac12 \lbr \int_{S_1} \iprod{\PP\nbar u}{\nbar u} \iprod{\PP \vec{e_1}}{\vec{e_1}} \,d\sigma - \int_{B_1} \lbr \sum_{i}^n p_{ii}\rbr \iprod{\PP\nbar u }{\nbar u}\,dx 
				+ 2\int_{B_1} \iprod{\PP\nbar u}{\PP\nbar u} -\err \rbr.
			\end{multline*}
			In particular, after simplification, we get
			Simplifying, we get
			
			\begin{multline}\label{a4.17}
				\lbr \frac{1}{2\ve}- \frac{\ve}{2(n-1)\la\La}\rbr \frac12 \err + 	\int_{B_1}\lbr 1 - \lbr \frac{1}{2\ve}- \frac{\ve}{2(n-1)\la\La}\rbr \frac{(2\La - \lbr \sum_{i}^n p_{ii}\rbr )}{2}\rbr \iprod{\PP(x)\nbar u}{\nbar u}\, dx \leq \\
				\lbr \frac{\ve}{2(n-1)\la} + \lbr \frac{1}{2\ve}- \frac{\ve}{2(n-1)\la\La}\rbr \frac{\La}{2}\rbr \int_{S_1} \iprod{\PP \nbar u}{\nbar u} \,d\sigma.
			\end{multline}
			\item[Step 5:]
			Ignoring the contribution from $\err$ momentarily (see \cref{rmka4.9}), we proceed as follows:
			Denoting $T := \sum_i^n p_{ii}$, we see that $n\la \leq T \leq n \La$. So, let us maximize \cref{a4.17}  over $\ve \leq \sqrt{(n-1)\la\La}$ and $n\la \leq T \leq n\La$ to get
			\begin{equation*}
				\int_{B_1}\max_{T\in[n\la,n\La]}\max_{\ve\in[0, \sqrt{(n-1)\la\La}]}\frac{\lbr 1 - \lbr \frac{1}{2\ve}- \frac{\ve}{2(n-1)\la\La}\rbr \frac{(2\La - T )}{2}\rbr}{\lbr \frac{\ve}{2(n-1)\la} + \lbr \frac{1}{2\ve}- \frac{\ve}{2(n-1)\la\La}\rbr \frac{\La}{2}\rbr} \iprod{\PP\nbar u}{\nbar u}\, dx \leq 
				\int_{S_1} \iprod{\PP \nbar u}{\nbar u} \,d\sigma.
			\end{equation*}
			This maximum is achieved when $T=n\La$ and 
			\[
			\ve  = \frac{(2-n)\La + \sqrt{(n-2)^2\La^2 + 4(n-1)\la\La}}{2} = \frac{\La}{2} \lbr -(n-2)+ \sqrt{(n-2)^2+ \frac{4(n-1)\la}{\La}}\rbr.
			\]
			Furthermore, it is easy to see that $\ve\in[0, \sqrt{(n-1)\la\La}]$. 
		\end{description}
		Substituting the values of $\ve$ and $T$  into \cref{a4.17} gives
		\begin{equation*}
			\lbr \sqrt{(n-2)^2+ \tfrac{4(n-1)\la}{\La}}\rbr\int_{B_1} \iprod{\PP\nbar u}{\nbar u}\, dx \leq 
			\int_{S_1} \iprod{\PP \nbar u}{\nbar u} \,d\sigma,
		\end{equation*}
		which completes the proof of the theorem.
	\end{proof}
	
	\begin{remark}\label{rmka4.9}
		If we were to keep track of $\err$, then we would get
		\begin{multline*}
			\err \frac{\frac12\lbr \frac{1}{2\ve}- \frac{\ve}{2(n-1)\la\La}\rbr}{\lbr \frac{\ve}{2(n-1)\la} + \lbr \frac{1}{2\ve}- \frac{\ve}{2(n-1)\la\La}\rbr \frac{\La}{2}\rbr}
			+ \int_{B_1}\frac{\lbr 1 - \lbr \frac{1}{2\ve}- \frac{\ve}{2(n-1)\la\La}\rbr \frac{(2\La - T )}{2}\rbr}{\lbr \frac{\ve}{2(n-1)\la} + \lbr \frac{1}{2\ve}- \frac{\ve}{2(n-1)\la\La}\rbr \frac{\La}{2}\rbr} \iprod{\PP\nbar u}{\nbar u}\, dx 
			\leq 
			\int_{S_1} \iprod{\PP \nbar u}{\nbar u} \,d\sigma.
		\end{multline*}
		which after  substituting the values of $\ve$ and $T$, becomes
		\begin{equation*}
			\frac{(n-2)}{\La\lbr \sqrt{(n-2)^2+ \tfrac{4(n-1)\la}{\La}}\rbr} \err + 		\lbr \sqrt{(n-2)^2+ \tfrac{4(n-1)\la}{\La}}\rbr\int_{B_1} \iprod{\PP\nbar u}{\nbar u}\, dx \leq 
			\int_{S_1} \iprod{\PP \nbar u}{\nbar u} \,d\sigma.
		\end{equation*}
		In particular, for general $\A(x)$ matrices, \cref{mainthm2} holds provided $\err \geq 0$. 
	\end{remark}
	
	\begin{question}
		Does the following inequality hold for some $\mathbf{C}_o >0$?
		\[
		\mathbf{C}_o\err \geq 	\lbr \sqrt{(n-2)^2+ \tfrac{4(n-1)\la}{\La}}\rbr\int_{B_1} \iprod{\PP\nbar u}{\nbar u}\, dx - \int_{S_1} \iprod{\PP \nbar u}{\nbar u} \,d\sigma.
		\]
		
		If the above inequality holds, then in the case $\err \leq 0$, we automatically get the required estimate, thus completing the proof.
		
	\end{question}

	\section*{References}

\end{document}